\documentclass[12pt, twoside, leqno]{article}

\usepackage{amsmath,amsthm}
\usepackage{amssymb}
\usepackage[T1]{fontenc}
\usepackage{enumerate}

\newtheorem{theorem}{Theorem}
\newtheorem{proposition}[theorem]{Proposition}	
\newtheorem{remark}[theorem]{Remark}
\newtheorem{lemma}[theorem]{Lemma}
\newtheorem{definition}[theorem]{Definition}
\newtheorem{corollary}[theorem]{Corollary}

\def\SU{{\bf SU}}
\def\Sp{{\bf Sp}}
\def\U{{\bf U}}
\def\SO{{\bf SO}}

\DeclareMathOperator{\tr}{tr}
\DeclareMathOperator{\Ad}{Ad}
\DeclareMathOperator{\diag}{diag}

\def\a{\mathfrak{a}}
\def\kk{\mathfrak{k}}
\def\p{\mathfrak{p}}
\def\g{\mathfrak{g}}

\def\F{{\bf F}}
\def\R{{\bf R}}
\def\C{{\bf C}}
\def\C{{\bf C}}
\def\H{{\bf H}}

\def\GL{{\bf GL}}
\def\SU{{\bf SU}}

\begin{document}


\baselineskip=17pt


\title{A Laplace-type representation for some generalized spherical functions of type $BC$}

\author{P.\ Sawyer\\
Department of mathematics and computer science\\ 
Laurentian University\\
Sudbury, Ontario\\
Canada P3E 2C6\\
E-mail: psawyer@laurentian.ca
}

\date{}

\maketitle


\renewcommand{\thefootnote}{}



\renewcommand{\thefootnote}{\arabic{footnote}}
\setcounter{footnote}{0}


\begin{abstract}
In \cite{Voit}, R\"osler and Voit give a formula for generalized spherical functions of type $BC$ in terms of the
spherical functions of type $A$.  We use this formula to describe precisely the support of the associated generalized Abel transform.
Furthermore, we derive a similar formula for the generalized spherical functions in the rational Dunkl setting.  The support of the intertwining operator $V$ is also deduced.

We also show, as a consequence, that a Laplace-type expression exists for the generalized spherical functions both in the trigonometric Dunkl setting and in the rational Dunkl setting.
\end{abstract}

\section{Introduction}

We introduce here the background for this paper. We refer the reader to Helgason's books \cite{Helgason1, Helgason2} as the standard reference on symmetric spaces.  Let $G$ be a semisimple Lie group of noncompact type with maximal 
compact subgroup $K$, let $\g=\kk\oplus\p$ be a Cartan decomposition of $G$ and let $\a$ be a maximal Abelian subspace of $\p$.
The corresponding symmetric space of noncompact type is $M=G/K$.   We also recall the definition of the Cartan motion group and the symmetric space of Euclidean type associated with $G$: it is the semi-direct product $G_0=K\rtimes \p$ where the group multiplication is defined by $(k_1,X_1)\cdot(k_2,X_2)=(k_1\,k_2,\Ad(k_1)(X_2)+X_1)$ while the associated symmetric space of Euclidean type is then $M_0=\p\simeq G_0/K$ (the action of $G_0$ on $\p$ is given by $(k,X)\cdot Y=\Ad(k)(Y)+X$).

We recall some notions concerning spherical functions, the Abel transform and its dual.  For $X\in \a$ and $\lambda\in \a_\C$, the spherical functions in the noncompact case are given by the equation
\begin{align*}
\phi_\lambda(e^X)=\int_K\,e^{(i\,\lambda-\rho)(H(e^X\,k))}\,dk
\end{align*} 
where $\rho(k)=\frac{1}{2}\,\sum_{\alpha\in R_+}\,k_\alpha\,\alpha$, $k_\alpha$ is the multiplicity of $\alpha$ and $g=k\,e^{H(g)}\,n\in K\,A\,N$ refers to the Iwasawa decomposition and by 
\begin{align*}
\psi_\lambda(X)=\int_K\,e^{i\,\lambda(\pi_\a(\Ad(k)\cdot X))}\,dk
\end{align*}
for the corresponding Cartan motion group (Euclidean type) where $\pi_\a$ is the orthogonal projection from $\p$ to $\a$ with respect to the Killing form.

When $X\not=0$, the spherical functions have a Laplace-type representation 
\begin{align}
\phi_\lambda(e^X)&=\int_{\a}\,e^{i\,\langle\lambda,H\rangle}\,K(H,X)\,dH,\label{LT}\\
\psi_\lambda(X)&=\int_{\a}\,e^{i\,\langle\lambda,H\rangle}\,K_0(H,X)\,dH\nonumber
\end{align}
with $K(H,X)\geq0$ and  $K_0(H,X)\geq0$ and where the support of $K(\cdot,X)$ and $K_0(\cdot,X)$ is $C(X)$, the convex hull of $W\cdot X$ in $\a$. Observe that $X\not=0$ ensures that $\dim C(X)=\dim\a$ by \cite[Theorem 10.1, Chap.{} IV]{Helgason2}.   

The function $K(H,\cdot)$ is the kernel of the Abel transform
\begin{align*}
\mathcal{A}(f)(e^H)=e^{\rho(H)}\,\int_N\,f(a\,n)\,dn=\int_{\a}\,f(e^X)\,K(H,X)\,\delta(X)\,dX
\end{align*}
where $\delta(X)=\prod_{\alpha\in R_+}\,\sinh^{m_\alpha}\alpha(X)$ while the dual Abel transform is simply given by
\begin{align}
\mathcal{A}^*(f)(e^X)=\int_K\,f(e^{H(e^X\,k)})\,dk=\int_{\a}\,f(e^H)\,K(H,X)\,dH\label{Dual}
\end{align}
so that $\phi_\lambda=\mathcal{A}^*(e^{i\,\langle\lambda,\cdot\rangle})$. 

We can also define the Abel transform $\mathcal{A}_0$  and its dual $\mathcal{A}_0^*$ for the symmetric space of Euclidean type in a similar fashion
(the area element $\delta(X)$ then becomes $\pi(X) =\prod_{\alpha\in R_+}\,\alpha(X)^{m_\alpha}$).

In Figure \ref{Dunkl}, we find some of the basic objects that will come into our discussion on the trigonometric Dunkl operators (also called Cherednik operators) and the rational Dunkl operators.  These operators provide the basis for the generalization of the spherical functions and the Abel transform to root systems with arbitrary multiplicities. The reader should refer to \cite{Anker,Opdam1,Roesler3,Sawyer1} for more details.
\begin{figure}[ht]
\begin{center}
\begin{tabular}{|p{6.0cm}|p{6.0cm}|}\hline
{\bf trigonometric Dunkl setting}  & {\bf rational Dunkl setting}\\ \hline
generalization of the symmetric space of noncompact type&generalization of the symmetric space of Euclidean type \\ \hline
\parbox[t]{6.0cm}{$\displaystyle D_\xi=\partial_\xi+\sum_{\alpha\in R_+}\,k_\alpha\,\alpha(\xi)\,\frac{1-r_\alpha}{1-e^{-\alpha}}$ ~\qquad${}-\rho(k)(\xi)$}
&$\displaystyle T_\xi=\partial_\xi+\sum_{\alpha\in R_+}\,k_\alpha\,\alpha(\xi)\,\frac{1-r_\alpha}{\langle \alpha,X\rangle}$\\\hline
\multicolumn{2}{|l|}{\parbox[t]{12.0cm}{$\langle\cdot,\cdot\rangle$ denotes the Killing form, $R_+$ is the set of positive roots, $\partial_\xi$ is the derivative in the direction of $\xi\in\a$, $r_\alpha(X)=X-2\,\frac{\langle \alpha,X\rangle}{\langle \alpha,\alpha\rangle}\,\alpha$, $\alpha\in R$, $X\in\a$ and $W$ is the group generated by the $r_\alpha$'s}}\\\cline{1-2}
$D_\xi\,E(\lambda,\cdot) = \langle\xi,\lambda\rangle\,E(\lambda,\cdot)$,\hfill\break $\lambda\in \a_\C^*$;  unique analytic solution with $E(\lambda,0)=1$
&$T_\xi\,E(\lambda,\cdot) = \langle\xi,\lambda\rangle\,E(\lambda,\cdot)$,\hfill\break  $\lambda\in \a_\C^*$; unique analytic solution with $E(\lambda,0)=1$\\\hline
$p(D_\xi)\,J(\lambda,\cdot) = p(\langle\xi,\lambda\rangle)\,J(\lambda,\cdot)$, $\lambda\in \a_\C^*$; unique analytic solution with $J(\lambda,0)=1$
&$p(T_\xi)\,J(\lambda,\cdot) = p(\langle\xi,\lambda\rangle)\,J(\lambda,\cdot)$, $\lambda\in \a_\C^*$; unique analytic solution with $J(\lambda,0)=1$\\ \cline{1-2}
$p(\partial_\xi)\circ\mathcal{A}=\mathcal{A}\circ p(D_\xi)$ , $p$ any symmetric polynomial
&$p(\partial_\xi)\circ\mathcal{A}=\mathcal{A}\circ p(T_\xi)$, $p$ any symmetric polynomial\\ \hline
$\mathcal{A}^*\circ p(\partial_\xi)=p(D_\xi)\circ \mathcal{A}^*$ , $p$ any symmetric polynomial 
&$\mathcal{A}^*\circ p(\partial_\xi)=p(T_\xi)\circ \mathcal{A}^*$ , $p$ any symmetric polynomial\\ \hline
\multicolumn{2}{|l|}{$J(\lambda,X)=\frac{1}{|W|}\,\sum_{s\in W}\,E(\lambda,s\cdot X)$, $X\in\a$}\\\cline{1-2}
\multicolumn{2}{|l|}{$J(w\cdot\lambda, w'\cdot X)=J(\lambda,X)$, $X\in\a$ and $w$, $w'\in W$}\\\hline
\end{tabular}
\caption{Dunkl setting and trigonometric Dunkl setting\label{Dunkl}}
\end{center}
\end{figure}

The functions $J(\lambda,\cdot)$ generalize the spherical functions on the symmetric spaces of noncompact type to arbitrary multiplicities (in the trigonometric Dunkl setting) and those on the symmetric spaces of Euclidean type (in the rational Dunkl setting). In the coming sections, we will use the notation $\phi_\lambda$ in the trigonometric setting and $\psi_\lambda$ in the rational setting.

As for the function $E(\lambda,\cdot)$ in the rational Dunkl setting, we have the following representation
\begin{align*}
E(\lambda,\cdot)&=V\,e^{\langle\lambda,\cdot\rangle}
\end{align*}
where $V$ is called the Dunkl intertwining operator.   We also introduce the positive measure $\mu_x$ such that
\begin{align*}
V_X(f)=\int_{\a}\,f(H)\,d\mu_X(H)
\end{align*}
(for the existence of the positive measure, see for example \cite{Roesler2}).  The support of $V_X$ is known to contain $W\cdot X$ and to be included in  $C(X)$ 
(refer to \cite{Anker} for example).

The intertwining operator $V$ has the property that 
$T_\xi\,V=V\,\partial_\xi$.  Compare with the intertwining properties of the generalized Abel transform and of its dual given in Figure \ref{Dunkl}.

Recently (\cite{Trimeche4}), Trim\`eche was able to introduce the intertwining operator in the trigonometric Dunkl setting.

The terms ``group case'' or ``geometric setting'' refer to the situations corresponding to the roots systems associated to symmetric spaces. 

From now on, unless stated otherwise, we are considering root systems of type $BC$. 
When referring to the root systems of type $A$ which come into play in our discussions, we will use the superscript $A$.  The integer $q\geq 1$ will be assumed to be fixed as well as $d=\dim_{\R}\,\F$ where $\F=\R$, $\C$ or $\H$.

In the geometric setting, the space $\a$ is made of the matrices
\begin{align*}
\left[
\begin{array}{ccc}
0_{q\times q}&D_X&0_{q\times (p-q)}\\
D_X&0_{q\times q}&0_{q\times (p-q)}\\
0_{(p-q)\times q}&0_{(p-q)\times q}&0_{(p-q)\times (p-q)}
\end{array}
\right]
\end{align*}
where $D_X=\diag[x_1,\dots,x_q]$ with $p\geq q$.  If $p>q$ then $\a^+=\{X\in\a\colon x_1>x_2>\dots>x_p>0\}$; if $p=q$ then $\a^+=\{X\in\a\colon x_1>x_2>\dots>|x_p|\}$.  We will assume that we are in the former case.

Note that the Killing form $\langle X,Y\rangle=C\,\tr X\,Y$ for some $C>0$ where $\tr$ denotes the trace.  This means that $\langle X,Y\rangle=2\,C\,\sum_{i=1}^q\,x_i\,y_i$.  

Figure \ref{geom} shows the multiplicities of the various roots depending on $p$, $q$ and $d$.

\begin{figure}[h]
\begin{center}
\begin{tabular}{|c|c|}\hline
$\alpha$&$m_\alpha$\\ \hline
$h_i$&$d\,(p-q)$\\ \hline
$2\,h_i$&$d-1$\\ \hline
$h_i-h_j$&$d$\\ \hline
$h_i+h_j$&$d$\\ \hline
\end{tabular}
\end{center}
\caption{The root system $BC$\label{geom} for the symmetric spaces $\SO_0(p,q)/\SO(p)\times\SO(q)$ ($d=1$), $\SU(p,q)/\SU(p)\times\SU(q)$ ($d=2$) and
$\Sp(p,q)/\Sp(p)\times\Sp(q)$ ($d=4$) with $p\geq q$}
\end{figure}

We set some terminology here which differs slightly from the usage in \cite{Voit}.

\begin{definition}\label{chamber}
We write $C_q=\{X\colon X=\diag[x_1,\dots,x_q],~x_i\in\R\}$ and
$C_q^+=\{X\in C\colon x_1>x_2>\dots>x_q>0\}$ (in \cite{Voit}, $C_q$ is used for $\overline{C_q^+}=\{X\in C_q\colon x_1\geq x_2\geq \dots\geq x_q\geq 0\}$).

The Weyl group $W$ acts on $C_q$ by permuting the $x_i$'s and by changing any numbers of its signs.  We also define $C(X)$ as the convex hull of $W\cdot X$ in $C_q$.

\end{definition}

The following result from \cite{Voit} is a generalization of a result from \cite{Sawyer2} which was given for the symmetric spaces 
$\SO_0(p,q)/\SO(p)\times\SO(q)$. 

\begin{theorem}[\cite{Voit}]\label{phiBC}
The generalized spherical function associated to the root system of type $BC$ is given as 
\begin{align}
\phi_\lambda(X)&=\int_{B_q}\,\phi^A_\lambda(Z(X,w)) \,|\det(Z(X,w))|^{-d\,(p+1)/2+1}\,dm_p(w)\,dw\label{phi}
\end{align}
where $\Re p> 2\,q-1$, $B_q=\{w\in M_q(\F)\colon I-w^*\,w>0\}$, $Z(X,w)=\cosh X+\sinh X\,w$, 
$m_p(w)=\frac{1}{\kappa^{p\,d/2}}\,\det(I-w^*\,w)^{p\,d/2-d\,(q-1/2)-1}$,
$\kappa^{p\,d/2}=\int_{B_q}\,\det(I-w^*\,w)^{p\,d/2-d\,(q-1/2)-1}\,dw$ and $\phi^A_\lambda$ is a spherical function for $\GL_0(q,\F)$ (the connected component of the $q\times q$ non-singular matrices over $\F$).
\end{theorem}

This formulation is adapted from \cite{Voit} to use the formula for the spherical functions of type $A$ (refer to \cite{Sawyer3} for example) and the formulas
$\rho^{BC}=\sum_{i=1}^q\,(\frac{d}{2}\,(p+q+2-2\,i)-1)\,f_i$ and $\rho^{A}=\sum_{i=1}^q\,\frac{d}{2}\,(q+1-2\,i)\,f_i$ where $f_i(X)=x_i$.

This allows us to provide the following expression for the dual of the Abel transform:
\begin{align}\label{Abel}
\mathcal{A}^*(f)(X)&=\int_{B_q}\,(\mathcal{A}^A)^*(f)(Z(X,w))
\,|\det(Z(X,w))|^{-d\,(p+1)/2+1}\,dm_p(w)\,dw.
\end{align}

Some additional notation which will be used throughout the paper:
\begin{definition}\label{XZ}
In addition to the notation $Z(X,w)=\cosh X+\sinh X\,w$, we will also denote $Z_2(X,w)=Z(X,w)^*\,Z(X,w)$ and $X(w)=a^A(Z(X,w))$ (the logarithms of the singular values of $Z(X,w)$ which equal half the logarithms of the eigenvalues of $Z_2(X,w)$ in decreasing order).
\end{definition}

This paper is organized as follows. We will start by assembling some preliminary result in Section \ref{prelim}.
In Theorem \ref{main} of Section \ref{Dually}, we will show that the support of the measure $f\mapsto \mathcal{A}^*(f)(X)$ is $C(X)$ as in the geometric setting.  
We also show in Theorem \ref{KHX} that the dual of the Abel transform  has a kernel $H\mapsto K(H,X)$ provided $X\not=0$.  In Section \ref{DunklSetting}, we derive 
an expression similar to the one given in \eqref{phi} for the generalized spherical functions associated to the root system $BC$ in the rational Dunkl setting.  This allows us to adapt the results of Theorem \ref{main} and Theorem \ref{KHX} to that setting.  We conclude by showing that the support of the intertwining operator $f\mapsto V_X(f)$ is also $C(X)$.

In \cite{Trimeche3},  Trim\`eche proved the more general result for all root systems that the measures $\mu_X$ (and therefore the corresponding dual of the Abel transform) are absolutely continuous with respect to the measure on $\a$ provided that $X$ is regular.  His results apply both in the rational and in the trigonometric settings.  However, our results for the root system $BC$ provide a more direct construction of the density $K$ (for the dual of the Abel transform), its exact support and are valid whenever $X\not=0$, a natural extension of the geometric setting.
The reader should also also consider the papers \cite{Trimeche1, Trimeche2} by  Trim\`eche on the root system $BC_d$ and $BC_2$.

\section{Preliminaries}\label{prelim}

In this section, we will introduce a few technical lemmas which will allow us to better describe the support of the measure $f\mapsto\mathcal{A}^*(f)(X)$.

The following results will provide a description of the convex hull $C(X)$ of an element $X$ of the Cartan subspace $\a$ based on the geometric setting described in the Introduction.  These results are easily transposed in terms of the sets $C_q$ and $C_q^+$ given in Definition \ref{chamber}.

\begin{lemma}\label{C+}
Let ${}^+C=\{H\in\a\colon \langle X,H\rangle>0~\hbox{for all}~ X\in\a^+\}$.  Then
\begin{align*}
{}^+C&=\{H\in\a\colon h_1>0,h_1+h_2>0,\dots, h_1+\dots+h_q>0\}.
\end{align*}
\end{lemma}

\begin{proof}
Any $X\in\a^+$ can be written with $x_q=y_q$, $x_{p-1}=y_q+y_{q-1}$, \dots, $x_1=y_1+\dots+y_q$ where the $y_i$'s are strictly positive and arbitrary.  The inequality $\langle X,H\rangle>0$ is equivalent to 
\begin{align*}
0&<h_1\,x_1+\dots+h_q\,x_q\\
&=(h_1+\dots+h_q)\,y_q + (h_1+\dots+h_{q-1})\,y_{q-1}+\dots+h_1\,y_1
\end{align*}
which holds for all $X\in\a^+$ if and only if $h_1+\dots+h_{k}>0$ for $k=1$, \dots, $q$.
\end{proof}

We are now in a position to describe the set $C(X)$ in terms of inequalities.

\begin{proposition}\label{CBq}
Let $X\in\overline{\a^+}$.  Then $H\in C(X)$ if and only if 
\begin{align}
\sum_{k=1}^r\,|h_{i_k}|\leq \sum_{k=1}^r\,x_k~
\hbox{for any choice of distinct $i_1$, \dots, $i_r$, $1\leq r\leq q$}.
\label{CX}
\end{align}

Furthermore, $H\in C(X)^\circ$ if and only if all the inequalities in \eqref{CX} are strict.
\end{proposition}

\begin{proof}
According to \cite[Lemma 8.3 page 459]{Helgason2}, $C(X)\cap \a^+=(X- {}^+C)\cap \a^+$. Now, from Lemma \ref{C+},
 if $X\in\overline{\a^+}$, 
\begin{align*}
X-{}^+C&=\{H\in\a\colon h_1< x_1,h_1+h_2< x_1+x_2,\dots, h_1+\dots+h_q<x_1+\dots+x_q\}
\end{align*}
and therefore
\begin{align}
(X-{}^+C)\cap\a^+&=\{H\in\a\colon h_1>h_2>\dots>h_q>0,~ h_1< x_1,h_1+h_2<x_1+x_2,\dots,
\nonumber\\ &\qquad\qquad\qquad
h_1+\dots+h_q<x_1+\dots+x_q\}.\label{CX2}
\end{align}
Let $D(X)$ be the set described in \eqref{CX}.  Since $D(X)$ is Weyl invariant and $D(X)\cap\a^+$ is equal to the 
set in \eqref{CX2}, the result follows.
\end{proof}

We recall a few results related to the root systems of type $A$ in order to exploit the relationship between $\phi_\lambda$ and $\phi_\lambda^A$ in \eqref{phi}.

Recall that $\a^A$ is the space of the  $q\times q$ real diagonal matrices,
that $(\a^A)^+$ is the set of the  $q\times q$ real diagonal matrices with strictly decreasing diagonal entries
and that the Cartan decomposition of $g\in \GL_0(q,\F)$ 
is given by $g=k_1\,e^{a^A(g)}\,k_2$ with  $k_i\in \U(q,\F)$ and 
$a^A(g)\in\overline{(\a^A)^+}=\{H=\diag[h_1,\dots,h_q]\in \a^A\colon h_i\geq h_{i+1}\}$.

\begin{remark}[\cite{Rado}]\label{CXA}
If $X=\diag[x_1,\dots,x_q]\in\overline{(a^A)^+}$ then $H=\diag[h_1,\dots,h_q]\in\a^A$ belongs to $C^A(X)$ if and only if $\sum_{i=1}^q\,h_i=\sum_{i=1}^q\,x_i$ and
\begin{align}
h_{i_1}+\dots+h_{i_k}&\leq x_1+\dots+x_k\label{ineq}
\end{align}
for every choice of distinct indices $i_1$, \dots, $i_k\in\{1,\dots,n\}$,   $1\leq k\leq n-1$.
Moreover, $H\in C^A(X)^\circ$ (in the relative topology of the set $\{H\colon \tr H=\tr X\}$) if and only if all the inequalities in \eqref{ineq} are strict.
\end{remark}

The next result is helpful in giving a ``smooth criteria'' to determine the number of distinct eigenvalues of a matrix.

\begin{theorem}[Hermite]\label{companion}
Let $P$ be a polynomial with real coefficients: the number of distinct
roots of $P$ is equal to the rank of the matrix
\begin{align*}
B= \left(\begin{array}{cccc}p_0&p_1&\ldots&p_{n-1}\\
p_1&p_2&\ldots&p_{n}\\
\vdots&\vdots&\ddots&\vdots\\
\vdots&\vdots&\ddots&\vdots\\
p_{n-1}&p_n&\ldots&p_{2n-2}\\
\end{array}\right)
\end{align*}
where $p_k=\sum b_j^k$ where $b_1$,\dots, $b_n$ are the roots of $P$.
The Newton polynomials $p_k$ are polynomials of the coefficients
of $P$.
\end{theorem}

\begin{proof}
Refer to \cite[Page 202]{Gantmacher}.
\end{proof}

The next result will be used repeatedly.

\begin{corollary}\label{eigen}
Let $f\colon B_q\to \hbox{Symm}(q,\F)$ (the space of $q\times q$ Hermitian matrices over $\F$) be analytic and, for $1\leq r\leq q$, let 
\begin{align*}
U_r=\{w\in B_q\colon \hbox{$f(w)$ has at least $r$ distinct eigenvalues}\}.
\end{align*}
Then $U_r$ is an open set of $B_q$ and either $U_r=\emptyset$ or $U_r$ is dense in $B_q$.
\end{corollary}

\begin{proof}
Let $B_w$ be the matrix associated to the polynomial $P=\det(t\,I_q-f(w))$ by the theorem.  Let $F\colon B_q\to\R$ be such that $F(w)$ is 
the sum of the squares of the absolute values of all the $r\times r$ sub-determinants of $B_w$ and note that $F$ is analytic.  
From the theorem, $U_r=F^{-1}((0,\infty))$ which is open; either $F$ is identically equal to 0 and $U_r=\emptyset$ or $U_r$ is dense in $B_q$.
\end{proof}

\section{The dual of the Abel transform}\label{Dually}

We will describe precisely the support of the dual of the Abel transform in Subsection \ref{Support} and then show in Subsection \ref{LaplaceSec} that, provided $X\not=0$, the measure $f\mapsto \mathcal{A}^*(f)(X)$ has a density.

\subsection{The support of the dual of the Abel transform}\label{Support}

We start by recalling the invariance properties of the Abel transform.

\begin{remark}\label{WINV}
We can observe directly from \eqref{Abel} that $\mathcal{A}^*(f)(s\cdot X)=\mathcal{A}^*(f)(X)$ for every $s\in W$.  Indeed, we have $s=\Ad(k)\circ L(u)$ where $k\in \U(q,\F)$ and $\Ad(k)$ acts on the elements of $C_q$ by permuting the diagonal elements and $\Ad(k)(H)=k\,H\,k^*$ and $u=\diag[u_1,\dots, u_q]$, $u_i\in\{-1,1\}$ with $L(u)(H)=u\,H$.  It then suffices to observe that $Z(s\cdot X,w)=\Ad(k)(Z(X,L(u)\,\Ad(k^*)\,w))$, that $dw$ is invariant under the action of the map $L(u)\,\Ad(k^*)$  and that $\Ad(k)\in W^A$, the Weyl group for the root system $A_{q-1}$.

Note also that $\mathcal{A}^*(f\circ s)=\mathcal{A}^*(f)$ for every $s\in W$ (a direct consequence of \eqref{LT} and \eqref{Dual} and of the last two lines in the table of Figure \ref{Dunkl}).  This implies that the support of the measure $f\mapsto  \mathcal{A}^*(f)(X)$ is Weyl-invariant.
\end{remark}

The next result provides a decomposition of $C(X)$.

\begin{proposition}\label{XC}
Let $X\in\a$ and $q\geq1$.  Then $C(X)=\cup_{w\in\overline{B_q}}\,C^A(X(w))$.
\end{proposition}

\begin{proof}
Suppose for now that $p\geq 2\,q$ is an integer (we are then in the geometric setting).  
For $k=\left[\begin{array}{cc}U&0\\0&V\end{array}\right]\in K=\SU(q,\F)\times\SU(p,\F)$, 
let $w(k)$ be the $q\times q$ principal minor of $V$.  It is not difficult to show that under the above conditions on $p$, 
$\overline{B_q}=\{w=w(k)\colon k\in K\}$.  Indeed, if $w=k_1\,\diag[\sigma_1,\dots,\sigma_q]\,k_2\in \overline{B_q}$, then it suffices to take $k$ with $U\in \SU(q,\F)$ arbitrary and 
$V=\left[\begin{array}{cc}A&B\\C&D\end{array}\right]\,\diag[1,\dots,1,\alpha]$ with $A=w$, $B=k_1\,\Sigma_B$ where $\Sigma_B$ is a $q\times(p-q)$ matrix with $(\Sigma_B)_{ii}=\sqrt{1-\sigma_i^2}$ and zero elsewhere, $C=\Sigma_C\,k_2$ where $\Sigma_C$ is a $(p-q)\times $ matrix with $(\Sigma_C)_{ii}=-\sqrt{1-\sigma_i^2}$ and zero elsewhere, $D=\diag[\sigma_1,\dots,\sigma_q,\overbrace{1,\dots,1}^{p-q}]$ and $|\alpha|=1$ is chosen so that $\det V=1$.  This proves that $\overline{B_q}\subseteq\{w=w(k)\colon k\in K\}$.  The reverse inclusion is straightforward.

Given the definition of $X(w)$ and of $B_q$, it is therefore sufficient to prove the proposition in the group case with $p\geq 2\,q$.

With $k$ as above, $H(e^X\,k)=H^A((\cosh D_X+\sinh D_X\,w(k)\,U^*)\,U)\in C^A(X(w(k)\,U^*))$ (\cite{Sawyer2}).  Since $C(X)=\{H(e^X\,k)\colon k\in K\}$ this implies that $C(X)\subseteq\cup_{w\in\overline{B_q}}\,C^A(X(w))$.  On the other hand, let $w_0\in \overline{B_q}$; we have 
$\cosh D_X+\sinh D_X\,w_0=k_1\,e^{X(w_0)}\,k_2$ with $k_i\in \U(q,\F)$, $i=1$, 2.  Hence, 
$X(w_0)=H^A(k_1^*\,(\cosh D_X+\sinh D_X\,w_0)\,k_2^*)=H^A((\cosh D_X+\sinh D_X\,w_0)\,k_2^*)=H^A((\cosh D_X+\sinh D_X\,(w_0\,k_2^*)\,k_2)\,k_2^*)
=H(e^X\,k)$ where $k=\left[\begin{array}{cc}k_2^*&0\\0&V\end{array}\right]$ and $V$ is built as above from $w=w_0\,k_2^*$.  This shows that
$X(w_0)\in C(X)$ and therefore that $C^A(X(w_0))\subseteq C(X)$ since $C(X)$ is $W^A$ invariant and convex.
\end{proof}

In what follows, $C^A(X(w))^\circ$ means the interior of $C^A(X(w))$ in the relative topology of $\{H\in\a^A\colon \tr H=\tr X(w)\}$.

\begin{theorem}\label{main}
Let $q\geq 1$ be an integer and suppose $p > 2\,q-1$ (we assume that $p$ is real here). Then the support of $f\mapsto \mathcal{A}^\ast(f)(X)$ is $C(X)$.
\end{theorem}

\begin{proof}
If $X=0$ then $\mathcal{A}^*(f)(X)=f(X)$ and the support is $\{X=0\}=C(X)$.  Based on Remark \ref{WINV}, we can assume that $X\in\overline{C_q^+}\setminus\{0\}$. 

If $q=1$ in \eqref{Abel}, then
\begin{align*}
\lefteqn{\mathcal{A}^*(f)(X)=\int_{B_1}\,f(\log(|\cosh x_1+\sinh x_1\,w|))}
\\&\qquad\qquad\qquad\qquad
\,\cdot|\cosh x_1+\sinh x_1\,w|^{-d\,(p+1)/2+1}\,dm_p(w)\,dw;
\end{align*}
$\log(|\cosh x_1+\sinh x_1\,w|)$ takes the full range between $-x_1$ and $x_1$ for $w\in \overline{B_1}=\{w\colon |w|\leq1\}$.  The result follows in this case.

Assume now that $q\geq2$ and use the notation of Lemma \ref{BqX}.
We have
\begin{align*}
\hbox{support}(\mathcal{A}_X^*)
=\overline{\cup_{w\in B_q}\,C^A(X(w))^\circ}
=\cup_{w\in \overline{B_q}}\,C^A(X(w))=C(X).
\end{align*}
The first equality follows from \eqref{Abel} and the fact that the density of the measure $f\mapsto(\mathcal{A}^A)^*(f)$ is positive on $C^A(X(w))$ (\cite{Sawyer1}).  Now, let $H_0\in C(X)$ and pick $\epsilon>0$.  There exists $H\in C(X)\setminus \R\,I_q$ such that $\|H-H_0\|<\epsilon/3$.  By Proposition \ref{XC}, $H=\sum_{s\in W^A}\,a_s\,s\cdot X(w)$ with $0\leq a_s\leq 1$, $\sum_{s\in W^A}\,a_s=1$ and $w\in \overline{B_q}$.  Note that since $H\not\in\R\,I_q$, the same is true of $X(w)$.  Since  the map $w\mapsto X(w)$ is continuous on $\overline{B_q}$, there exists $w'\in B_q$ with $\|X(w')-X(w)\|<\epsilon/3$ and $X(w')\not\in \R\,I_q$ (this is necessary to ensure that $\overline{C^A(X(w'))^\circ}=C^A(X(w'))$).  Hence, $H'=\sum_{s\in W^A}\,a_s\,s\cdot X(w')$ satisfies $\|H'-H\|<\epsilon/3$.  Finally, there exists $H''\in C^A(X(w'))^\circ$ such that $\|H''-H'\|<\epsilon/3$ and therefore $\|H''-H_0\|<\epsilon$. 
The result follows.
\end{proof}

We will end this subsection by describing more closely the interior of $C(X)$ which will be useful when showing that the density of the measure 
$f\mapsto \mathcal{A}^*(f)(X)$ is strictly positive on $C(X)^\circ$.

\begin{definition}
For $X\in \overline{C_q^+}$, let $\mathcal{U}(X)=\{\diag[u_1,\dots,u_q]\colon u_1\leq u_2\leq\dots\leq u_q,~|u_i|<x_i~\hbox{if $x_i>0$ and $u_i=0$ otherwise}\}$.  Note that $U=X(w)\in\mathcal{U}(X)$ where $w=\diag[y_1,\dots,y_q]\in B_q$ with $y_i=(e^{u_i}-\cosh x_i)/\sinh x_i$ if $x_i\not=0$ and $y_i$ arbitrary in the interval $(-1,1)$ if $x_i=0$.
\end{definition}

The definition of the set $B_q(X)$ in the next lemma will allow us to avoid the difficulty that arises when $C^A(X(w))$ consists of only one point.

\begin{lemma}\label{BqX}
For $X\in C_q$, define $f_X\colon B_q\to \R$ by $f_X(w)=\tr X(w)$ and $B_q(X)=\{w\in B_q\colon \hbox{$df_X$ is surjective at $w$ and $X(w)\not\in\R\,I_q$}\}$. If $q\geq 2$ and $X\in C_q\setminus\{0\}$ then $B_q(X)$ is open and dense in $B_q$.  Furthermore, $\mathcal{U}(X)\setminus\R\,I_q\subseteq B_q(X)$.  
\end{lemma}

\begin{proof}
We can assume without loss of generality that $X\in \overline{C_q^+}\setminus\{0\}$ (refer to the discussion in Remark \ref{WINV}).
Observe that we can write $f_X(w)=\log(\det(Z_2(X,w)))/2$ and that $f_X$ and therefore $df_X$ are analytic on $B_q$.  Let $V_X=\{w\in B_q\colon \hbox{$\left.df_X\right|_w$ is surjective}\}$.  Since the rank of $\left.df_X\right|_w$ is a matter of a determinant being nonzero, $V_X$ is open and is either empty or dense in $B_q$.  Taking 
$w=\diag[y_1,\dots,y_q]\in \mathcal{U}(X)$, we have 
$\left.df_X\right|_w(E_{1,1})=\left.\frac{d~}{dt}\right|_{t=0}\,f_X(w+t\,E_{1,1})
=\sinh x_1/(\cosh x_1+y_1\,\sinh x_1)\not=0$.  Hence $V_X$ is a dense open set in $B_q$.

Note also that $w\in B_q(X)$ is equivalent to $w\in V_X$ and $\dim C^A(X(w))=q-1$ (this is the maximum it can be since $H\in C^A(X(w))$ implies $\tr H=\tr X(w)$).
Observe that $w\in B_q(X)$ if and only if $w\in V_X$ and $f(w)=Z_2(X,w)$ is not a multiple of the identity.  Consider now Corollary \ref{eigen} with $r=2$ and observe that 
$B_q(X)=V_X\cap U_2$.  Since $X\in \overline{C_q^+}\setminus\{0\}$ then either $X=\mu\,I_q$, $\mu>0$ or $X$ has at least two distinct diagonal entries.  In the first case, we observe that
$w=\diag[\mu/2,-\mu/2,0\dots,0]\in B_q(X)=V_X\cap U_2$.  In the second case, if $X$ has at least two distinct non-negative diagonal entries, then $w=0\in B_q(X)=V_X\cap U_2$.  In both cases, the result follows from Corollary \ref{eigen}.
\end{proof}

The following technical result will be useful in the proof of Proposition \ref{CC} below.

\begin{lemma}\label{special}
Suppose $X\in\overline{C_q^+}\setminus\{0\}$.  Let $w_0=\diag[y_1,\dots,y_q]\in B_q(X)$ be such that
$X(w_0)=\diag[u_1,\dots,u_q]$ with $u_i\geq u_{i+1}$ for all $i$.  Then for $b>0$ small enough, $w_0+b\,E_{1,q}\in B_q(X)$ and $X(w_0+b\,E_{1,q})=\diag[u_1+\delta,u_2,\dots,u_{q-1},	u_q-\delta]$ for some $\delta>0$
\end{lemma}

\begin{proof}
This situation can be reduced to considering the case when $X=\diag[x_1,x_2]$, $w_0=\diag[y_1,y_2]$ and $X(w_0)=\diag[u_1,u_2]$.  One only has to look at the eigenvalues of $Z_2(w_0+b\,E_{1,2})=\left[ \begin {array}{cc} e^{2\,u_1}&e^{u_1}
\sinh  x_1 \,b\\ e^{u_1}
\sinh  x_1 \,b&e^{2\,u_2}+\sinh^2x_1\,b^2\end {array} \right] $ which is elementary.
\end{proof}

The following result will allow us to show that the density of the measure $f\mapsto \mathcal{A}^*(f)(X)$ is strictly positive on $C(X)^\circ$ when $X\not=0$.

\begin{proposition}\label{CC}
Let $q\geq 2$.
For $X\in C_q$,  $\cup_{w\in B_q(X)}\,C^A(X(w))^\circ\subseteq C(X)^\circ$ and $C(X)^\circ\cap\overline{C_q^+}\subseteq\cup_{w\in B_q(X)}\,C^A(X(w))^\circ$.
\end{proposition}

\begin{proof}
Assume that $X\in \overline{C_q^+}$ and that $X\not=0$ (if $X=0$ then the result is straightforward).

We first show that if $H'\in C^A(X(w_0))^\circ$ and $w_0\in B_q(X)$ then $H'\in C(X)^\circ$.  Suppose that this is not the case: we have 
$|h'_{i_1}|+\dots+|h_{i_r}'|=x_1+\dots+x_r$ for some $r$ and distinct indices $i_k$. if $r<q$, let $I=\{i_k\colon h'_{i_k}\geq 0\}$,
$J=\{i_k\colon h'_{i_k}< 0\}$ and define $f_{I,J}(H)=\sum_{i\in I}\,h_i-\sum_{j\in J}\,h_j$.  Since $f_{I,J}$ is linear and not constant on 
$C^A(X(w_0))$ which is convex and compact, we have
\begin{align*}
|h'_{i_1}|+\dots+|h'_{i_r}|=f_{I,J}(H')<\max_{H\in C^A(X(w))}\,f_{I,J}(H)=f_{I,J}(H'')
\end{align*}
for some $H''\in C^A(X(w))$.  This means that $|h''_{i_1}|+\dots+|h''_{i_r}|\geq f_{I,J}(H'')>f_{I,J}(H')=x_1+\dots+x_r$ which contradicts $H''\in C(X)$.
If $r=q$ and not all $h_i$'s are of the same sign, the same reasoning applies.  Otherwise, if all $h'_i$'s are of the same sign, the map $f_X$ of Lemma \ref{BqX} would reach a maximum or a minimum at $w_0$ which is impossible since $\left.df_X\right|_w\not=0$ for all $w\in B_q(X)$.

We now show that $C(X)^\circ\cap\overline{C_q^+}\subseteq\cup_{w\in B_q(X)}\,C^A(X(w))^\circ$.  Pick $H\in C(X)^\circ\cap\overline{C_q^+}$.  
If $H=0$, let $y_i=(1-\cosh x_i)/\sinh x_i$ (if $x_i\not=0$) and $y_i\in (-1,1)$ arbitrary when $x_i=0$ and set $w_0=\diag[y_1,\dots,y_q]$.  We have $H=X(w_0)$ and, using Lemma \ref{special}, we have $H\in C^A(X(w_0+b\,E_{1.q}))^0$ for $b>0$ small enough. From the proof of Lemma \ref{BqX}, we conclude that $w_0+b\,E_{1.q}\in B_q(X)$ provided again that $b$ is small enough.

Suppose then that $H\not=0$ and let $j$ be the smallest index $k$ such that $\sum_{i=1}^k\,x_i> \sum_{i=1}^q\,h_i$.  Let $U$ be defined by the relations 
\begin{align*}
u_k&=\left\lbrace
\begin{array}{cl}
x_k-\epsilon&\hbox{if $k< j$}\\
(j-1)\,\epsilon+\sum_{i=1}^q\,h_i-\sum_{i=1}^{j-1}\,x_i&\hbox{if $k=j$}\\
0&\hbox{if $j<k \leq q$}\\
\end{array}
\right.
\end{align*}
where $0<\epsilon<\min\{x_k, 1\leq k\leq j-1,(\sum_{i=1}^{r}\,x_i-\sum_{i=1}^r\,h_i)/r,r=1,\dots,j,(\sum_{i=1}^{j}\,x_i-\sum_{i=1}^q\,h_i)/(j-1)\}$.

We easily verify the inequalities $u_1\geq u_2\geq\dots\geq u_q\geq 0$, $u_i< x_i$, when $x_i>0$ using the definition of $j$ and the restrictions on $\epsilon$.   We have $U=X(w_0)\in\mathcal{U}(X)$ and since $u_{j-1}
-u_j\geq \sum_{i=1}^j\,x_i-\sum_{i=1}^q\,h_i-j\,\epsilon>0$, $U\not\in \R\,I_q$ and therefore $w_0\in B_q(X)$ by Lemma \ref{BqX}.  

We verify that $H\in C^A(U)$ using Remark \ref{CXA}: 
since $U$ and $H\in\overline{(\a^A)^+}$,  we only have to show that $h_1+\dots+h_q= u_1+\dots+u_q$ which is straightforward and that
$h_1+\dots+h_r\leq u_1+\dots+u_r$ for every $r<q$.  For $r< j$, this follows immediately from the definition of $u_i$ and $\epsilon$ and when $j\leq r<q$,
we have $h_1+\dots+h_r\leq h_1+\dots+h_q=u_1+\dots+u_{r}$.

Using Lemma \ref{special} with $U=X(w_0)$ and noting that $H\in C^A(X(w_0+b\,E_{1,q}))^\circ$ with $w_0+b\,E_{1,q}\in B_q(X)$ provided $b$ is small enough, we can conclude.  
\end{proof}

\subsection{A Laplace-type expression for the generalized spherical functions}\label{LaplaceSec}

As mentioned in the Introduction, in the geometric setting, the measure $f\mapsto \mathcal{A}^*(f)(X)$ is absolutely continuous with respect to the 
Lebesgue measure on $\a$ provided $X\not=0$.  We now show that this remains true in the present context.

\begin{theorem}\label{KHX}
Let $q\geq 1$ be an integer and suppose $p > 2\,q-1$ (we assume that $p$ is real here).  There exists a measurable non-negative function $K$ defined on $C_q\times (C_q\setminus\{0\})$ such that for every measurable function $f$ on $C_q$ and every 
$X\in C_q\setminus\{0\}$, we have 
\begin{align*}
\mathcal{A}^*(f)(X)
&=\int_{C(X)^\circ}\,f(H)\,K(H,X)\,dH
\end{align*}
where $K(H,X)>0$ for all $H\in C(X)^\circ$.  In particular, the support of the map $H\mapsto K(H,X)$ is $C(X)$. 
\end{theorem}

\begin{proof}
If $q=1$, then the result follows from the beginning of the proof of Theorem \ref{main}.  We may therefore assume that $q\geq2$.

Let $K^A$ be the kernel of the Abel transform for the root systems of type $A$ (this was generalized to arbitrary multiplicities in 
\cite[Theorem 2.3]{Sawyer1} but here we are only using the group case).  Recall that provided $X\not\in\R\,I_q$,
\begin{align*}
(\mathcal{A}^A)^*(f)(X)
&=\int_{C^A(X)^\circ}\,f(H)\,K^A(H,X)\,dH
\end{align*}
with $K^A(\cdot,X)>0$ on $C^A(X)^\circ$.  Let $X(w)=a^A(Z(X,w))$ as before and assume $q\geq 2$.  We have
\begin{align*}
\mathcal{A}^*(f)(X)
&=\int_{B_q}\,(\mathcal{A}^A)^*(f)(Z(X,w))\,|\det(Z(X,w))|^{-d\,(p+1)/2+1}\,dm_p(w)\,dw\\
&=\int_{B_q(X)}\,(\mathcal{A}^A)^*(f)(Z(X,w))\,|\det(Z(X,w))|^{-d\,(p+1)/2+1}\,dm_p(w)\,dw\\\\
&=\int_{B_q(X)}\,\int_{C^A(X(w))^\circ}\,f(e^H)\,K^A(H,X(w))\,dH\,\det(e^{X(w)})^{-d\,(p+1)/2+1}\,dm_p(w)\,dw\\
&=\int_{\cup_{w\in B_q(X)}\,C^A(X(w))^\circ}\,f(e^H)\,\left[\int_{D_H(X)}\,K^A(H,X(w))\,\det(e^{X(w)})^{-d\,(p+1)/2+1}\,dm_p(w)\,dw\right]\,dH
\end{align*}
where $D_H(X)=\{w\in B_q(X)\colon H\in C^A(X(w))^\circ\}$. Now,
\begin{align}
K(H,X)&=\int_{D_H(X)}\,K^A(H,X(w))\,\det(e^{X(w)})^{-d\,(p+1)/2+1}\,dm_p(w)\,dw\nonumber\\
&=\frac{1}{\kappa^{p\,d/2}}\,\det(e^{H})^{-d\,(p+1)/2+1}\label{DHX}
\\&\qquad\qquad
\,\cdot\int_{D_H(X)}\,K^A(H,X(w))\,\det(I-w^*\,w)^{p\,d/2-d\,(q-1/2)-1}\,dw.\nonumber
\end{align}

Suppose now that $H\in C(X)^\circ\cap \overline{C_q^+}$ and consider the open set $U(X,H)=\{w\in B_q(X)\colon \sum_{i=1}^r\,x_k(w)>\sum_{k=1}^r\, h_k,~1\leq r\leq q-1\}$.  The set $D_H(X)$ is the nonempty intersection of $U(X,H)$ and the submanifold $\{w\in B_q(X)\colon \tr X(w)=\tr H\}$ (refer to Proposition \ref{CC}).  For every $w_0\in D_H(H)$, there exists a coordinate system $(U,\phi)$ where $U$ is an open subset of $U(X,H)$ containing $w_0$, $\phi\colon U\to\R^{d\,q^2}$ and $\phi(w)=(x_1,\dots,x_{d\,q^2-1},0)$ on $U\cap D_H(X)$.  This allows us to integrate over $D_H(X)$.  Therefore, the map $H\to K(H,X)$ is strictly positive on $C(X)^\circ\cap \overline{C_q^+}$ and by the Weyl invariance of the map, over all of $C(X)^\circ$.
\end{proof}

\begin{corollary}\label{LaplaceType}
For $X\not=0$, we have
\begin{align*}
\phi_\lambda(X)
&=\int_{C(X)^\circ}\,f(H)\,K(H,X)\,dX.
\end{align*}
\end{corollary}

\begin{remark}
Note that the formulas in \eqref{DHX} remain valid with some adjustment when $q=1$; when in addition $d=1$, the set $D_H(X)$ only contains one point and the integral over that set disappears.

Note also that the integral term in the last expression for $K(H,X)$ (second line of \eqref{DHX}) in the proof of the theorem is decreasing with $p$ and corresponds to a group case when $p\geq 2\,q$ is an integer.  Since in the geometric setting, the support is known to be $C(X)$, we could also have deduced the same result from that observation in the more general case.

Note also that the function $K(H,X)$ in the theorem is still defined when $\Re p>2\,q-1$ (\emph{i.e.} when $p$ is not assumed to be real) but $K$ is no longer real and proving that its support is exactly $C(X)$ is another matter.
\end{remark}

\section{The rational Dunkl setting}\label{DunklSetting}

In this section, we will derive results which provide the counterparts of formulas \eqref{phi} and \eqref{Abel} as well as of the results of Theorem \ref{main}, Theorem \ref{KHX} and its corollary in the rational Dunkl setting.

In \cite{Sawyer1}, we use the following result (originally from \cite{Sawyer3}) about the spherical functions associated to the root systems of type $A$ in the trigonometric Dunkl setting:
\begin{theorem}\label{old}
For $X\in (\a^A)^+$, we define $\phi_\lambda^A(X)=e^{i\,\lambda(X)}$ when $q=1$ and for 
$q\geq 2$,
\begin{align*}
\phi^A_\lambda(X)
=\frac{\Gamma(d\,q/2)}{(\Gamma(d/2))^q}
e^{i\,\lambda_q\,\sum_{k=1}^q\,x_k}
\int_{E(X)}\, \phi^A_{\lambda_0}(e^\xi)\,S^{(d)}(\xi,X)\,d(\xi)^d\,d\xi
\end{align*}
where $E(X)
=\{\xi=(\xi_1,\dots,\xi_{p-1})\colon x_{k+1}\leq \xi_k\leq x_k\}$, $\lambda(X)=\sum_{j=1}^q\,\lambda_j\,x_j$,
$\lambda_0(\xi)=\sum_{i=1}^{q-1}\,(\lambda_i-\lambda_q)\,\xi_i$, $d(X)=\prod_{r<s}\,\sinh(x_r-x_s)$, $d(\xi)=\prod_{r<s}\,\sinh(\xi_r-\xi_s)$ and
\begin{align*}
S^{(d)}(\xi,X)=d(X)^{1-d}\,d(\xi)^{1-d}\,\left[\prod_{r=1}^{q-1}
\,\left(\prod_{s=1}^r\,\sinh(x_s-\xi_r)
\,\prod_{s=r+1}^q\,\sinh(\xi_r-x_s)\right)\right]^{d/2-1}
\end{align*}
\end{theorem}
\noindent and prove the following in the rational Dunkl setting:
\begin{theorem}\label{psidunkl}
Let $X\in(\a^A)^+$.  The generalized spherical function associated to the root system $A_{q-1}$ in the rational Dunkl setting is given by 
$\psi_\lambda^A(X)=e^{i\,\lambda(X)}$ when $q=1$ and for 
$q\geq 2$,
\begin{align*}
\psi_\lambda^A(X)
&=\frac{\Gamma(d\,q/2)}{\Gamma(d/2)^q}
\,e^{i\,\lambda_q\,\sum_{k=1}^r\,x_k}
\,\int_{E(X)}\,\psi^A_{\lambda_0}(e^\xi)\,T^{(d)}(\eta,X)\,d_0(\eta)^d\,d\eta
\end{align*}
where $X\in (\a^A)^+$, $E(X)$, and $\lambda_0$ are as before,  $d_0(\xi)=\prod_{r<s}\,(\xi_r-\xi_s)$, $d_0(X)=\prod_{r<s}\,(x_r-x_s)$ and
\begin{align*}
T^{(d)}(\xi,X)&=
d_0(X)^{1-d} \,d_0(\xi)^{1-d}
\prod_{r=1}^q\,\left[
\prod_{s=1}^{r-1}\,(x_r-\xi_s) \,\prod_{s=r}^{q-1}\,(\xi_s-x_r)
\right]^{d/2-1}.
\end{align*}
\end{theorem}

\begin{remark}\label{intT}
The results given in \cite{Sawyer1} are actually extended to $X\in\overline{(\a^A)^+}$.
Note also that if $\sigma=\{(\beta_1,\dots,\beta_q)\colon \beta_i\geq0,\sum_{i=1}^q\,\beta_i=1\}$ then 
\begin{align*}
\frac{\Gamma(d\,q/2)}{\Gamma(d/2)^q}
\,\int_{E(X)}\,T(\eta,X)\,d_0(\eta)^d\,d\eta=\frac{\Gamma(d\,q/2)}{\Gamma(d/2)^q}\,\int_\sigma (\beta_1\cdots\beta_q)^{d/2-1}\,d\beta=1.
\end{align*}

This follows directly by making the change of variable $\beta_k=\frac{\prod_{i=1}^{q-1}\,(\xi_i-x_k)}{\prod_{i\not=k}\,(x_i-x_k)}$,
$k=1$, \dots, $q$.
\end{remark}

The next result describes the behaviour of $a^A(Z(\epsilon\,X,w))$ when $\epsilon$ tends to 0; a step necessary to apply the technique of ``rational limits'' by de Jeu (\cite[Theorem 4.13]{DeJeu}) later on.

\begin{lemma}\label{He}
Let $X\in C_q$.
For $w\in \overline{B_q}$, write $Z(\epsilon\,X,w)=k_1(\epsilon)\,e^{X^\epsilon(w)}\,k_2(\epsilon)$ with $X^\epsilon(w)\in\overline{(\a^A)^+}$
and $k_i(\epsilon)\in\U(q,\F)$. Then
\begin{align}
\lim_{\epsilon\to0}\,\frac{X^\epsilon(w)}{\epsilon}=a^A\left(\exp\left(\frac{X\,w+w^*\,X}{2}\right)\right)\label{limit}
\end{align}
uniformly on $\overline{B_q}$. 
\end{lemma}

\begin{proof}
Note that the logarithm function is analytic when defined on the space of positive definite matrices with values in the space of symmetric matrices.
We have $Z_2(\epsilon\,X,w)=k_1(\epsilon)\,e^{2\,X^\epsilon(w)}\,k_1(\epsilon)^*$; observe also that $X^\epsilon(w)$ is continuous in $w$ and $\epsilon$.
Assuming that $|\epsilon|<1$ and $w\in\overline{B_q}$,
\begin{align*}
\det(t\,I-\frac{X^\epsilon(w)}{\epsilon})
&=\det(t\,I-\frac{k_1(\epsilon)\,2\,X^\epsilon(w)\,k_1^*(\epsilon)}{2\,\epsilon})
=\det(t\,I-\frac{\log(k_1(\epsilon)\,e^{2\,X^\epsilon(w)}\,k_1^*(\epsilon))}{2\,\epsilon})\\
&=\det(t\,I-\frac{\log(Z_2(\epsilon\,X,w))}{2\,\epsilon}).
\end{align*}

As $\epsilon$ tends to 0, the last term converges uniformly to $\det(t-(X\,w+w^*\,X)/2)$ since,
provided $|\epsilon|<<1$, 
\begin{align*}
\log(Z_2(\epsilon\,X,w))&=\log(I+[Z_2(\epsilon\,X,w)-I])
=\sum_{k=1}^\infty\,(-1)^{k+1}\,\frac{[Z_2(\epsilon\,X,w)-I]^k}{k}\\
&=\epsilon\,(X\,w+w^*\,X)/2+O(\epsilon^2)
\end{align*}
which follows from the Taylor expansion of $Z_2(\epsilon\,X,w)=I+\epsilon\,(X\,w+w^*\,X)/2+O(\epsilon^2)$. 

Since the characteristic polynomial of $X^\epsilon(w)/\epsilon$ tends uniformly to the characteristic polynomial of $(X\,w+w^*\,X)/2$ as $\epsilon\to0$, this proves
\eqref{limit}.
\end{proof}

\begin{corollary}
There exists $\delta>0$ and a compact set $\tilde{E}(X)$ independent of $w\in\overline{B_q}$ and $\epsilon$ such that $E(a^A\left(\exp\left(\frac{X\,w+w^*\,X}{2}\right)\right))
\subseteq\tilde{E}(X)$ and $E(X^\epsilon(w)/\epsilon)
\subseteq\tilde{E}(X)$ whenever $0<\epsilon<\delta$.
\end{corollary}

\begin{theorem}\label{psiBC}
The generalized spherical function associated to the root system of type $BC$ in the rational Dunkl setting is given as 
\begin{align*}
\psi_\lambda(X)&=\int_{B_q}\,\psi^A_\lambda(\exp((X\,w+w^*\,X)/2))\,dm_p(w)\,dw
\end{align*}
where  
$\psi^A_\lambda$ is the spherical function for the symmetric space of Euclidean type associated to $\GL_0(q,\F)$.
\end{theorem}

\begin{proof}
We will assume that $X\in C_q^+$ (the result follows for all $X\in\a$ by continuity and Weyl-invariance following a reasoning similar to the one in Remark \ref{WINV}). We first show that 
\begin{align}
\lim_{\epsilon\to0}\,\phi_{\lambda/\epsilon}^A(\cosh (\epsilon\,X)+\sinh (\epsilon\,X)\,w)
=\psi_\lambda^A(e^{(X\,w+w^*\,X)/2})\label{ae}
\end{align}
almost everywhere and that
\begin{align}
|\phi_{\lambda/\epsilon}^A(\cosh (\epsilon\,X)+\sinh (\epsilon\,X)\,w)|
\leq M+|\psi_\lambda^A(e^{(X\,w+w^*\,X)/2})|\label{bounded}
\end{align}
for some constant $M>0$ independent of $\epsilon$ and $w\in B_q$.
Write $a^A(\cosh (\epsilon\,X)+\sinh (\epsilon\,X)\,w)=X^\epsilon(w)=\diag[x_1^\epsilon(w),\dots,x_q^\epsilon(w)]$ and
let $\dot{X}^0(w)=\lim_{\epsilon\to0}\,X^\epsilon(w)/\epsilon=a^A\left(\exp\left(\frac{X\,w+w^*\,X}{2}\right)\right)=\diag[\dot{x}^0_1,\dots,\dot{x}^0_q]$ from Lemma \ref{He}.  The set $U=\{w\in B_q\colon \frac{X\,w+w^*\,X}{2}~\hbox{has distinct eigenvalues}\}$ is open and dense in $B_q$ by Corollary \ref{eigen} (taking $r=q$ and noting that $w=I_q/2\in U=U_q$).

For $w\in U$, provided $\epsilon$ is close enough to 0, the diagonal entries of $X^\epsilon(w)/\epsilon$ are distinct and therefore $X^\epsilon(w)\in(\a^A)^+$.
From Theorem \ref{old}, we have
\begin{align*}
\lefteqn{\phi^A_{\lambda/\epsilon}(\cosh (\epsilon\,X)+\sinh (\epsilon\,X)\,w)}\\
&=\frac{\Gamma(d\,q/2)}{(\Gamma(d/2))^q}\,
e^{i\,\lambda_q\,\sum_{k=1}^q\,x_q^\epsilon(w)/\epsilon}
\int_{E(X^\epsilon(w))}\, \phi^A_{\lambda_0/\epsilon}(e^\xi)\,S^{(d)}(\xi,X^\epsilon(w))\,d(\xi)^d\,d\xi\\
&=\frac{\Gamma(d\,q/2)}{(\Gamma(d/2))^q}
e^{i\,\lambda_q\,\sum_{k=1}^q\,x_q^\epsilon(w)/\epsilon}
\int_{E(X^\epsilon(w)/\epsilon)}\, \phi^A_{\lambda_0/\epsilon}(e^{\epsilon\,\eta})\,\left[\epsilon^{q-1}\,S^{(d)}(\epsilon\,\eta,X^\epsilon(w))
\,d(\epsilon\,\eta)^d\right]\,d\eta\\
&=\frac{\Gamma(d\,q/2)}{(\Gamma(d/2))^q}\,\int_{E(X^\epsilon(w)/\epsilon)}
\,F^{(d)}(\lambda_0,\epsilon,\eta,X^\epsilon(w)/\epsilon)
\,T^{(d)}(\eta,X^\epsilon(w)/\epsilon)\,d_0(\eta)\,d\eta
\end{align*}
where
\begin{align*}
F^{(d)}(\lambda_0,\epsilon,\eta,X^\epsilon(w)/\epsilon)
&=e^{i\,\lambda_q\,\sum_{k=1}^q\,x_q^\epsilon(w)/\epsilon}\,\prod_{i<j}\,\left(\frac{\sinh(x_i^\epsilon(w)-x_j^\epsilon(w))}
{\epsilon\,(x_i^\epsilon(w)/\epsilon-x_j^\epsilon(w)/\epsilon)}\right)^{1-d}
\,\prod_{i<j}\,\left(\frac{\sinh(\epsilon\,\eta_i-\epsilon\,\eta_j)}{\epsilon\,(\eta_i-\eta_j)}\right)^{1-d}
\\&\qquad
\cdot\left[\prod_{r=1}^{q-1}
\,\left(\prod_{s=1}^r\,\frac{\sinh(x_s^\epsilon(w)-\epsilon\,\eta_r)}{\epsilon\,(x_s^\epsilon(w)/\epsilon-\eta_r)}
\,\prod_{s=r+1}^q\,\frac{\sinh(\epsilon\,\eta_r-x_s^\epsilon(w))}{\epsilon\,(\eta_r-x_s^\epsilon(w)/\epsilon)}\right)\right]^{d/2-1}
\\&\qquad
\cdot\phi_{\lambda_0/\epsilon}^A(e^{\epsilon\,\eta})
\end{align*}
which converges uniformly to $e^{i\,\lambda_q\,\sum_{k=1}^q\,\dot{x}^0_q(w)}\,\psi_{\lambda_0}^A(e^\eta)$ on $\overline{B_q}\times \tilde{E}(X)$ as $\epsilon$ tends to 0 (\cite[Theorem 4.13]{DeJeu}).  

Using Theorem \ref{psidunkl}, it follows that 
\begin{align*}
\lefteqn{|\phi^A_{\lambda/\epsilon}(\cosh (\epsilon\,X)+\sinh (\epsilon\,X)\,w)-\psi_{\lambda_0}^A(e^{(X\,w+w^*\,X)/2})|}\\
&=\frac{\Gamma(d\,q/2)}{(\Gamma(d/2))^q}\,\left|\int_{E(X^\epsilon(w)/\epsilon)}
\,F^{(d)}(\lambda_0,\epsilon,\eta,X^\epsilon(w)/\epsilon) \,T^{(d)}(\eta,X^\epsilon(w)/\epsilon)\,d_0(\eta)^d\,d\eta
\right.\\&{}\qquad\left.
-e^{i\,\lambda_q\,\sum_{k=1}^q\,\dot{x}^0_q(w)}\,\int_{E(\dot{X}^0(w))}\,\psi_{\lambda_0}^A(e^\eta)\,T^{(d)}(\eta,\dot{X}^0(w)	)\,d_0(\eta)^d\,d\eta\right|\\
&\leq\frac{\Gamma(d\,q/2)}{(\Gamma(d/2))^q}\,\int_{E(X^\epsilon(w)/\epsilon)}
\,\left|F^{(d)}(\lambda_0,\epsilon,\eta,X^\epsilon(w)/\epsilon)-e^{i\,\lambda_q\,\sum_{k=1}^q\,\dot{x}^0_q(w)}\,\psi_{\lambda_0}^A(e^\eta)\right|
\\&\qquad\qquad\qquad\qquad\qquad\qquad\qquad\qquad \cdot T^{(d)}(\eta,X^\epsilon(w)/\epsilon)\,d_0(\eta)^d\,d\eta
\\&\qquad
+ \frac{\Gamma(d\,q/2)}{(\Gamma(d/2))^q}\,\left|e^{i\,\lambda_q\,\sum_{k=1}^q\,\dot{x}^0_q(w)}\right|\,\left|\int_{E(X^\epsilon(w)/\epsilon)}\,\psi_{\lambda_0}^A(e^\eta)\,T^{(d)}(\eta,X^\epsilon(w)/\epsilon)\,d_0(\eta)^d\,d\eta
\right.\\&\qquad\qquad\qquad\left.
      -\int_{E(\dot{X}^0(w))}\,\psi_{\lambda_0}^A(e^\eta)\,T^{(d)}(\eta,\dot{X}^0(w)	)\,d_0(\eta)^d\,d\eta\right|\\
&\leq M(\epsilon)\,\int_{E(X^\epsilon(w)/\epsilon)}\,T^{(d)}(\eta,X^\epsilon(w)/\epsilon)\,d_0(\eta)^d\,d\eta
\\ &\qquad
+C\,\left|\int_{\sigma}\,\psi_{\lambda_0}^A(e^{\eta(\beta(\epsilon))})\,(\beta_1\,\cdots\,\beta_q)^{d/2-1}\,d\beta
      -\int_{\sigma}\,\psi_{\lambda_0}^A(e^{\eta(\beta(0))})\,(\beta_1\,\cdots\,\beta_q)^{d/2-1}\,d\beta
		\right|\\
&\leq M(\epsilon)
+ C\,\int_{\sigma}\,\left|\psi_{\lambda_0}^A(e^{\eta(\beta(\epsilon))})-\psi_{\lambda_0}^A(e^{\eta(\beta(0))})\right|
\,(\beta_1\,\cdots\,\beta_q)^{d/2-1}\,d\beta
\end{align*}
(we used the change of variables of Remark \ref{intT} in the last inequalities) where $C>0$ is a constant independent of $\epsilon$ and $w\in\overline{B_q}$.

Now, $\lim_{\epsilon\to0}\,M(\epsilon)=0$ uniformly on $\overline{B_q}\times\tilde{E}(X)$ while $\exp(2\,\eta_i(\beta(\epsilon)))$, $i=1$, \dots, $q-1$, are the roots of the polynomial
$\sum_{r=1}^q\,\beta_r\,\prod_{i\not=r}\,(x-e^{2\,x_i^\epsilon(w)/\epsilon})$ and $\exp(2\,\eta_i(\beta(0)))$, $i=1$, \dots, $q-1$, are the roots of the polynomial 
$\sum_{r=1}^q\,\beta_r\,\prod_{i\not=r}\,(x-e^{2\,\dot{x}^0(w)_i})$.  From there, we conclude that the term 
\begin{align*}
\int_{\sigma}\,\left|\psi_{\lambda_0}^A(e^{\eta(\beta(\epsilon))})-\psi_{\lambda_0}^A(e^{\eta(\beta(0))})\right|
\,(\beta_1\,\cdots\,\beta_q)^{d/2-1}\,d\beta
\end{align*}
will also tend to 0 uniformly on $\overline{B_q}\times\sigma$ as $\epsilon$ tends to 0.  This proves \eqref{ae} and, noting that 
\begin{align*}
\lefteqn{|\phi^A_{\lambda/\epsilon}(\cosh (\epsilon\,X)+\sinh (\epsilon\,X)\,w)-\psi_{\lambda_0}^A(e^{(X\,w+w^*\,X)/2})|}\\
&\leq M(\epsilon)
+ C\,\int_{\sigma}\,\left|\psi_{\lambda_0}^A(e^{\eta(\beta(\epsilon))})-\psi_{\lambda_0}^A(e^{\eta(\beta(0))})\right|
\,(\beta_1\,\cdots\,\beta_q)^{d/2-1}\,d\beta,
\end{align*}
\eqref{bounded} follows.  Using \cite[Theorem 4.13]{DeJeu} once more, 
\begin{align*}
\psi_\lambda(X)&=\lim_{\epsilon\to0}\,\phi_{\lambda/\epsilon}(e^{\epsilon\,X})
\end{align*}
uniformly in $\lambda$ and $X$ over compact sets.  
Using \eqref{ae} and \eqref{bounded} and the fact that $\lim_{\epsilon\to0}\,X^\epsilon(w)=0$ uniformly on $\overline{B_q}$, the result follows easily by the dominated convergence theorem.
\end{proof}

The proof of the following theorem follows the same lines as in the trigonometric setting.

\begin{theorem}\label{main2}
Let $X\in\overline{C_q^+}$.  The support of the measure defined on $C_q$ by 
\begin{align*}
\mathcal{A}^*_0(f)(X)=\int_{B_q}\,(\mathcal{A}^A)^*_0(f)(\exp((X\,w+w^*\,X)/2)) \,dm_p(w)\,dw
\end{align*}
is $C(X)$.  Furthermore, if $X\not=0$, the measure $f\mapsto \mathcal{A}^*_0(f)(X)$ is absolutely continuous and its density is strictly positive on $C(X)^\circ$,
namely
\begin{align*}
K_0(H,X)&=\frac{1}{\kappa^{p\,d/2}}\,\int_{\widetilde{D_H(X)}}\,K_0^A(H,\dot{X}^0(w))\,\det(I-w^*\,w)^{p\,d/2-d\,(q-1/2)-1}\,dw
\end{align*}
where
\begin{align*}
\widetilde{D_H(X)}&=\{w\in \widetilde{B_q(X)}\colon H\in C^A(\dot{X}^0(w))^\circ\},\\
\widetilde{B_q(X)}&=\{w\in B_q\colon \hbox{$\dot{X}^0(w)\not\in \R\,I_q$ and $\left. d\widetilde{f_X}\right|_w$ is surjective} \},
\end{align*}
$\dot{X}^0(w)$ is the diagonal part of $(X\,w+w^*\,X)/2$ with decreasing diagonal entries and $\widetilde{f_X}\colon B_q\to \R$ is defined by 
$\widetilde{f_X}(w)=\tr\dot{X}^0(w)=\tr (X\,w+w^*\,X)/2$.
\end{theorem}

\begin{proof}
The proof is very similar to the one of Theorem \ref{KHX}.  We go over some of the differences. Recall that in the trigonometric setting,
\begin{align*}
\psi_\lambda(X)&=\int_K\,e^{i\,\langle \lambda,\Ad(k)\,X\rangle}\,dk
=\int_K\,e^{i\,\lambda(\pi_\a(\Ad(k)\,X))}\,dk
\end{align*}
where $\pi_\a\colon\p\to\a$ is the orthogonal projection with respect to the Killing form.  For the Lie groups $\SO_0(p,q)$, $\SU(p,q)$ and $\Sp(p,q)$, $p>q$, we have
\begin{align*}
\pi_\a(k\cdot X)&=\pi_\a\left(
\left[\begin{array}{cc}U&0\\0&V\end{array}\right]\,\left[\begin{array}{cc}0&D_X\,Q^T\\Q\,X&0\end{array}\right]\,\left[\begin{array}{cc}U&0\\0&V\end{array}\right]^*
\right)\\
&=\pi_\a\left(
\left[\begin{array}{cc}0&U\,X\,Q^T\,V^*\\V\,Q\,X\,U^*&0\end{array}\right]
\right)=\pi_\a\left(
\left[
\begin{array}{cc}
0&U\,X\,\lbrack A^*,C^*\rbrack\\
\lbrack A^*,C^*\rbrack\,X\,U^*&0
\end{array}
\right]
\right)\\
&=\pi_\a\left(
\left[\begin{array}{cc}0&H\,Q^T\\Q\,H&0\end{array}\right]
\right)
\end{align*}
where $Q^T=\left[\begin{array}{cc}I_q&0_{q\times(p-q)}\end{array}\right]$, $V=\left[\begin{array}{cc}A&B\\C&D\end{array}\right]$,
\begin{align*}
H&=\pi_{\a^A}(U\,X\,A^*)=\pi_{\a^A}((U\,X\,A^*+A\,X\,U^*)/2)=\pi_{\a^A}(U\,\frac{X\,(U^*\,A)^*+(U^*\,A)\,X}{2}\,U^*)\\
&=\pi_{\a^A}(U\,\frac{X\,w+w^*\,X}{2}\,U^*)=\pi_{\a^A}(U\,\dot{X}^0(w)\,U^*),
\end{align*}
and $\pi_{\a^A}(g)$ is the matrix made of the diagonal of $g$ and $w\in B_q$.  Now, the set
$\{\pi_{\a^A}(U\,\dot{X}^0(w)\,U^*)\colon U\in \U(q,\F)\}$ is the support of the measure $f\mapsto (\mathcal{A}_0^A)^*(X)$ which is equal to $C^A(\dot{X}^0(w))$ by \cite{Lu} .  These considerations lead us to $C(X)=\cup_{w\in \overline{B_q}}C^A(\dot{X}^0(w))$.  Hence, Proposition \ref{XC} (which relied also on the group case), Theorem \ref{main} and Proposition \ref{CC} have their counterparts from which the current result follows.
\end{proof}

\begin{corollary}
Let $X\in C_q^+$ and let $f\mapsto V_X(f)=\int_{C_q}\,f(H)\,d\mu_X(H)$ be the Dunkl intertwining operator.  Then the support of $V$ is also $C(X)$.
\end{corollary}

\begin{proof}
In \cite{Sawyer1} we showed, using a result from the doctoral thesis of C.{} Rejeb \cite[Theorem 2.9]{Rejeb} (also found in the paper \cite{Gallardo}) that the support of the measure $f\mapsto V_X(f)$ is the same as the support of the measure $f\mapsto \mathcal{A}^*(f)(X)$ (this proof did not rely on the specific root system).  The rest follows from the theorem.
\end{proof}

\begin{remark}
Given that  Trim\`eche has shown that the intertwining operator can be defined in the trigonometric setting, the corresponding result is also valid: for the root systems studied in this paper, the support of the measure $f\mapsto V_X(f)$ is also $C(X)$ in the trigonometric setting (refer to  \cite{Trimeche3, Trimeche4}).
\end{remark}

\section{Conclusion}

In this paper and in \cite{Sawyer1}, we have proved for specific root systems that the dual of the Abel transform, $H\mapsto \mathcal{A}^*(f)(X)$ has support exactly equal to $C(X)$ and that the same holds for the intertwining operator $V$ in the rational Dunkl setting.  We also now know that this transform has a kernel provided $X\not=0$.

These results were possible because we had an iterative formula for the generalized spherical function in the case of root systems of type $A$ and a formula ``reducing'' the problem to the root systems of type $A$ in the case of the root systems of type $BC$.

The drawback of this approach is that unless similar formulas are derived for the other root systems, it is not easily generalizable.

Another question of interest would be to see if \eqref{phi} and the results of this paper can be generalized in a setting where $d$ is no longer restricted to 1, 2 or 4.

\subsection*{Acknowledgments}
This research was partly supported by funding from Laurentian University.

\end{document}